\documentclass[12pt]{amsart}
\usepackage{amsfonts,amssymb,amscd,amsmath,enumerate,verbatim,calc}
\usepackage{tikz,tikz-cd,pgfplots}
\pgfplotsset{compat=1.18}
\usepackage{csquotes}
\usepackage{graphicx}
\usepackage{color,hyperref,comment,url}
\usepackage{mathrsfs}
\usepackage{fullpage}
\usepackage{ulem}
\usepackage{cite}
\usepackage{cleveref}
\usepackage{enumitem}

\newcommand{\bs}{\backslash}
\newcommand{\N}{\mathbb N}

\newcommand{\Z}{\mathbb Z}
\newcommand{\R}{\mathbb R}

\newcommand{\mbv}{\mathbf{v}}

\newcommand{\mfm}{\mathfrak{m}}

\renewcommand{\l}{\lambda}

\newcommand{\mb}[1]{\mathbf{#1}}
\newcommand{\quot}[1]{\widetilde{#1}}
\newcommand{\wk}[1]{{^*}#1}
\newcommand{\wkp}[1]{{^{*p}}\hspace{-2pt}#1}
\newcommand{\wkq}[1]{{^{*q}}\hspace{-2pt}#1}

\newcommand{\fte}{\operatorname{Fte}}
\newcommand{\fbp}[1]{\left[ #1 \right]}
\renewcommand{\phi}{\varphi}

\theoremstyle{definition}
\newtheorem{thm}{Theorem}[section]

\newtheorem{cor}[thm]{Corollary}
\newtheorem{prop}[thm]{Proposition}
\newtheorem{ex}[thm]{Example}
\newtheorem{rmk}[thm]{Remark}
\newtheorem*{defn*}{Definition}
\newtheorem{defn}[thm]{Definition}

\newtheorem{alg}{Algorithm}
\newtheorem{mainthm}{Theorem}

\definecolor{purple(x11)}{rgb}{0.63, 0.36, 0.94}

\renewcommand{\int}{\operatorname{int}}
\numberwithin{equation}{section}

\title{The weak normalization of an affine semigroup}
\author{Kyle Maddox \and Srishti Singh}
\date{May 15, 2025}
\thanks{2020  \textit{Mathematics Subject Classification}. Primary: 20M25, 13A35}

\thanks{Keywords: affine semigroups, weak normalization, Frobenius closure, Frobenius test exponent, $F$-injective, $F$-nilpotent}
\begin{document}
\begin{abstract}
    In this article, we define and explore the weak normalization of an affine semigroup. In particular, for a fixed prime integer, we provide a geometric description of the weak normalization of an affine semigroup with respect to that prime, which corresponds to the weak normalization of the affine semigroup ring over a field of that prime characteristic, similar to the description of the seminormalization of an affine semigroup given by Reid-Roberts. We then use this description to understand the singularities of an affine semigroup ring defined over a field of prime characteristic and provide several examples. In particular, we demonstrate that all affine semigroup rings defined over fields of prime characteristic have a uniform upper bound on the Frobenius test exponent of all ideals, which provides a large and important class of examples with a positive answer to a question of Katzman-Sharp on uniformity of Frobenius test exponents. Finally, we provide an algorithm and implementation to compute the weak normalization of an affine semigroup, as well as the Frobenius test exponent and Frobenius closures of ideals in the affine semigroup ring. 
\end{abstract}
\maketitle

\section{Introduction}

Our primary motivation in this article is to provide a computable adjustment of an affine semigroup to study the singularities of the corresponding affine semigroup ring defined over a field of prime characteristic. An affine semigroup is a finitely generated sub-monoid of $\mathbb{N}^n$ (or sometimes $\mathbb{Z}^n$). Beyond giving rise to rich combinatorial and geometric structure, they also provide a natural structure to study subrings of polynomial rings generated by monomials, called affine semigroup rings. Affine semigroups rings arise naturally in algebraic geometry as coordinate rings of toric varieties, which ``have provided a remarkably fertile testing ground for general theories" as is often-quoted from W. Fulton's famous work on the subject \cite{Fulton}.

One deficiency in the rich history of affine semigroup rings is that they have primarily been studied in the case that the affine semigroup is normal. Normal affine semigroups are especially well-behaved, essentially they have no gaps and so are more able to be studied from the perspective of convex geometry. As an example of this perspective, M. Hochster famously showed in \cite{Hochster} that a normal affine semigroup ring is a direct summand of a polynomial ring, which implies that it must be Cohen-Macaulay (and much more). 

A weaker notion than normality is \textit{seminormality}, and the \textit{seminormalization} (see \Cref{defn: seminormalization}) of an affine semigroup was studied by L. Reid and L. Roberts in \cite{Reid2001}, who provided a geometric characterization of the seminormalization. A seminormal affine semigroup ring can exhibit more subtle singularities than a normal affine semigroup ring, e.g. seminormal does not imply Cohen-Macaulay. While normality is strongly connected to the convex geometry of the cone of the semigroup, seminormality is a more algebraic condition. The primary condition we study in this article is in-between seminormality and normality.

Many primary references on the subject also focus on affine semigroup rings defined over the complex numbers. However, modern commutative algebra and algebraic geometry make extensive use of prime characteristic techniques, even in the study of rings of characteristic $0$ via reduction to prime characteristic. In particular, the study of singularities defined in terms of properties of the Frobenius map, called $F$-singularities, has exploded since the 1970's, as E. Kunz showed in \cite{Kunz} that a Noetherian ring of prime characteristic is regular if and only if the Frobenius map is flat.  

Of primary interest in this article are two $F$-singularities in particular; \textit{$F$-injective} and \textit{$F$-nilpotent}, see \Cref{defn: F-inj} and \Cref{defn: F-nil}. Both are defined in terms of the natural Frobenius action on local cohomology, and so one should expect these properties are difficult to verify for a particular ring. In the context of affine semigroup rings, however, work of Bruns-Li-R\"omer in \cite{Bruns2006} (for $F$-injective) and Dao-Maddox-Pandey in \cite{DMP} (for $F$-nilpotent) demonstrate that, for affine semigroup rings, these singularities are easier to verify by computational techniques, using associated primes of certain abelian groups coming from the semigroup and its cone in \cite{Bruns2006} and from properties of the normalization map in \cite{DMP}. To help unify these stories, we introduce the \textit{$p$-weak normalization} of an affine semigroup $A$, which we denote $\wkp{A}$. 

\begin{defn*}[\Cref{defn: p-weak normalization}]
    Let $A\subset \mathbb{N}^n$ be an affine semigroup. For a prime number $p$, define the \textbf{$p$-weak normalization of $A$} to be the set $\wkp{A} = \left\lbrace \left. \mb v \in \overline{A} \,\right|\, p^e \mb v \in A \text{ for some } e \in \N\right\rbrace.$
\end{defn*}

Fitting between the seminormalization $^{+}A$ and normalization $\overline{A}$, $\wkp{A}$ is another affine semigroup containing $A$ which helps capture the combinatorial information of the Frobenius map on the associated affine semigroup ring. However, unlike the normalization and seminormalization, the definition of weak normalization necessarily incorporates a fixed prime number and we show through several examples that as $p$ varies, $\wkp{A}$ can vary as well. For a fixed $A$, however, only finitely many distinct $p$-weak normalizations can exist, see \Cref{blr weakly normal}. 

Both the normalization and seminormalization have geometric description in terms of the group $G(A)$ and cone $C(A)$ of $A$ (see \Cref{subsec: affine semigroups}). We provide a similar geometric description of $\wkp{A}$, which utilizes an adjustment the quotient of an affine semigroup by a natural number, see \cite{Onthequotient}. In particular, for an affine semigroup $A$ and natural number $m$ we write $\quot{A/m} = \{ \mb v \in \overline{A} \mid m \mbv \in A\}$. 

\begin{mainthm}[\Cref{thm: geometric description of p-weak normalization}]
Let $A\subset \mathbb{N}^n$ be an affine semigroup and fix a prime $p$. Recall there is an $N_0$ such that $p^{N_0}(\wkp{A}) \subset A$. Then \[
        \wkp{A} = \bigcup_{\substack{F \text{ is a}\\ \text{face of } C(A)}} 
           \left( \quot{\frac{A \cap F}{p^{N_0}}} \cap \int(F)\right).
   \] 
\end{mainthm}

We also present an algorithm and implementation to compute $\wkp{A}$ given the generators of $A$ and $\overline{A}$ in \Cref{alg2}. 

Returning to the $F$-singularities side of the story, by combining the aforementioned results of Bruns-Li-R\"omer and Dao-Maddox-Pandey with the results in this article, we have the following characterization of $F$-injective and $F$-nilpotent singularities for affine semigroup rings based on our $p$-weak normalization.

\begin{mainthm}[\Cref{cor: wkNiffFinj}, \Cref{cor: weak normalization equals normalization if and only if F-nilpotent}]
Let $k$ be a field of prime characteristic $p>0$ and let $A$ be an affine semigroup. Then, $A=\wkp{A}$ if and only if $k[A]$ is $F$-injective and $\wkp{A}=\overline{A}$ if and only if $k[A]$ is $F$-nilpotent.
\end{mainthm}

Both the $F$-injective and $F$-nilpotent conditions are strongly connected to the Frobenius closure of ideals. Given an ideal $I$ in a ring $R$ of prime characteristic, the Frobenius closure $I^F$ is the set of elements $x \in R$ such that $F^e(x) \in F^e(I)R$. When $R$ is Noetherian, for each $I$ there must be an $e_0$ such that $F^{e_0}(I^F)R = F^{e_0}(I)R$, which is called the Frobenius test exponent $\fte I$ of $I$ (see \Cref{subsec: fsings}). In the introduction of \cite{Katzman-Sharp}, M. Katzman and R. Sharp ask whether the set $\{\fte I \mid I \subset R\}$ is bounded, and  from a computational perspective it is clearly desirable to obtain a uniform upper bound on $\fte I$ over all ideals $I$. 

However, in \cite{Brenner}, H. Brenner exhibited a family of standard graded normal domains which have the property that no uniform upper bound on $\fte I$ can exist. We show that affine semigroup rings are not susceptible to Brenner's problem. 

\begin{mainthm}[\Cref{thm: affine semigroup rings have uniformly bounded fte}]
Let $A$ be an affine semigroup and let $k$ be a field of prime characteristic $p>0$. Further, let $N_0$ be minimal so that $p^{N_0}(\wkp{A}) \subset A$. Then, for any ideal $I\subset k[A]$, $\fte I\le N_0$.
\end{mainthm}

To our knowledge, this result presents affine semigroup rings as the largest class of rings which have a positive answer to the question of Katzman-Sharp. A strong benefit to the computational approach in our article is that our algorithm to calculate $\wkp{A}$ also provides the uniform upper bound on $\fte I$ for ideals $I\subset k[A]$. We also provide an example (see \Cref{ex: 2 different weak normalizations part 3}) to show that, over $k=\mathbb{F}_p$, one can effectively compute Frobenius closures of any ideal in an affine semigroup ring $k[A]$ with our techniques. \\

\noindent\textbf{Acknowledgments:} We are very grateful to Kevin Tucker for numerous helpful conversations. Further, we want to thank Daniel Hader and Juan Ignacio Garc\'ia-Garc\'ia for discussions around implementing Algorithms \ref{alg1} and \ref{alg2}.

\section{Preliminaries}

Throughout this article, we assume all rings are Noetherian and excellent. For a reduced ring $R$, we will always use $\overline{R}$ to denote the normalization of $R$ in its total ring of quotients. Given a vector $\mb a = (a_1,\ldots,a_n) \in \mathbb{N}^n$, we will write $\mb x^{\mb a}$ for the monomial $x_1^{a_1}\cdots x_n^{a_n}$ in an ambient polynomial ring $R[x_1,\ldots,x_n]$.

\subsection{Affine semigroups}\label{subsec: affine semigroups}

In this subsection, we recall the concepts of affine semigroups, their quotients, and several related ideas from convex geometry, all of which play a role in describing the weak normalization of affine semigroup rings. Recall an \textbf{affine semigroup} is a finitely-generated sub-monoid of the additive monoid $\N^n$ for some positive integer $n$. 

Let $k$ be a field, and $A$ an affine semigroup generated by $\{\bf a_1,..., \bf a_s\} \subseteq \N^n$. The associated semigroup ring $k[A]$ is isomorphic to the $k$-algebra $k[x_1,...,x_s]/I$ where $I$ is the kernel of the ring homomorphism $\phi : k[x_1,\ldots,x_s] \to k[u_1,\ldots,u_n]$ defined by $x_i \mapsto \bf u^{\bf a_i}$ for $1 \leq i \leq s$. Since affine semigroup rings are finite type over $k$, they are both Noetherian and excellent.

An affine semigroup defines a convex rational polyhedral cone in a natural way, and therefore, techniques from convex geometry can be applied to study them. We will briefly recall these notions and direct the interested reader to Section 2 of \cite{Bruns2006} for further discussion. Denote by $G(A)$ the group generated by $\{\bf a_1,..., \bf a_s\}\subset \mathbb{Z}^n$, and $C(A) =\{r_1 {\bf a_1} +... r_s {\bf a_s} \mid r_1,...,r_s \in \R_{\geq 0} \}$ the cone of $A$. A face of $A$ is $F=\{r_1 {\bf a_{i_1}}+...+ r_u {\bf a_{i_u}} \mid r_1,...,r_u \in \R_{\geq 0}\}$ for some $\{i_1,...,i_u\} \subset \{1,...,s\}$. A vector $\mathbf{x}$ is in the interior $\int F$ of $F$ if $\mathbf x = r_1 {\bf a_{i_1}}+...+ r_u {\bf a_{i_u}}$ where $r_1,...,r_u \in \R_{> 0}$, which agrees with the usual topological interior on the induced subspace topology inherited from $\R^n$.

\begin{defn} \label{defn: saturation}
    Let $A$ be an affine semigroup. The \textbf{saturation} or \textbf{normalization} of $A$ is \[\overline{A} = \{\mathbf{v} \in G(A) \mid m \mathbf v \in A \text{ for some } m \in \N\} = G(A)\cap C(A).\]
\end{defn}

We end this subsection with the quotient of a semigroup by a natural number (see \cite{Onthequotient}), which will play a crucial role in the geometric description of the weak normalization of an affine semigroup. Note that the nomenclature ``quotient" is evocative of fractions rather than an algebraic quotient of a semigroup by a semigroup ideal.

\begin{defn}\label{defn: quotient}
    Let $A$ be an affine semigroup and fix $d \in \mathbb{N}$. The \textbf{quotient of $A$ by $d$} is the affine semigroup $A/d=\{\mathbf v \in \N^n \mid d \mbv \in A\}.$ We will be more interested in $\frac{A}{d}\cap G(A) = \{ \mb v \in \overline{A} \mid d\mb v \in A\}$, which we will write as $\quot{A/d}.$
\end{defn}

Notably, $\quot{A/d}\subset \overline{A}$ is itself an affine semigroup due to \cite[Cor.~2.10]{BG09}, since the cones of $A$ and $\quot{A/d}$ are the same. 

\subsection{Purely inseparable extensions and weak normalization}

As we define our $p$-weak normalization for affine semigroups, it is valuable to recall the algebraic situation. In this subsection, rings will be reduced and of prime characteristic $p>0$.

\begin{defn}
    An extension (inclusion) of rings $R\rightarrow S$ is \textbf{purely inseparable} if, as for fields, for each $s \in S$ there is an $e \in \mathbb{N}$ such that $s^{p^e}\in R$. Let $Q$ be the total ring of quotients of $R$; the largest purely inseparable extension of $R$ inside $Q$ is called the \textbf{weak normalization} of $R$, denoted $\wk{R}$. \[
    \wk{R} = \left\lbrace r \in Q \mid \text{there is an } e \in \N \text{ such that } r^{p^e} \in R\right\rbrace.
    \] The ring $R$ is \textbf{weakly normal} if $R=\wk{R}$.
\end{defn}

\begin{rmk}\label{rmk: normalization factorization diagram}
    Let $R$ be a reduced ring of prime characteristic $p>0$. We have a commutative diagram of ring inclusions \vspace{-.3in}\begin{center}\begin{equation}\label{wk normalization diagram}
        \begin{tikzcd}
           R \arrow{rr}{\iota}\arrow{dr}[swap]{\phi} & \, & \overline{R} \\
           \, & \wk{R}\arrow{ur}[swap]{\psi} &\,
        \end{tikzcd}\end{equation}
    \end{center} where $\iota$ is the usual normalization map, $\phi$ is purely inseparable, and $\psi$ is separable. Furthermore, by our initial assumption that $R$ is excellent, we have that each of these maps is a module-finite extension. Finally, since $R\rightarrow \wk{R}$ is module-finite and purely inseparable, there is an $e_0\in \N$ such that $r^{p^{e_0}}\in R$ for all $r \in \wk{R}$.
\end{rmk}

Our primary goal in this paper is to describe the weak normalization of affine semigroup rings in such a way as to be amenable to computation. A related notion of \textit{seminormalization} has been extensively studied before; see the excellent survey by M. Vitulli \cite{Vitulli2011} for details on the seminormalization and weak normalization in general. In the specific context of affine semigroup rings, works of Reid-Roberts \cite{Reid2001} and Bruns-Li-R\"omer \cite{Bruns2006} were particularly important as sources of inspiration for this article. 

\subsection{A brief primer on \texorpdfstring{$F$-}{F-}singularities}\label{subsec: fsings}

Throughout this subsection, all rings are of prime characteristic $p>0$. For such a ring $R$, the properties of the Frobenius endomorphism $F:R\rightarrow R$ defined $F(r)=r^p$ govern an extraordinary amount of information about $R$. In particular, Kunz famously showed in \cite{Kunz} that a Noetherian ring $R$ is regular if and only if $F:R \rightarrow R$ is flat. This justifies the study of singularities in prime characteristic as the study of the Frobenius map, and the use of the phrase $F$-singularities for singularities in prime characteristic. For more details on $F$-singularities in general, we direct the reader to  surveys by Schwede-Smith \cite{SchwedeSmith} and Ma-Polstra \cite{MaPolstra}.

We will focus on two particular $F$-singularities, both defined in terms of the Frobenius map on local cohomology; $F$-injective and $F$-nilpotent. For readers unfamiliar with local cohomology, we recommend Brunz-Herzog's book \cite{BH} and \textit{24 Hours of Local Cohomology} \cite{24hours}. For what follows, we will write $R^\circ$ for $R\setminus \bigcup_{\mathfrak{p} \in \operatorname{Min}(R)} \mathfrak{p}$, and if $R$ is a domain, note $R^\circ = R\setminus \{0\}$. 

\begin{defn}
Let $(R,\mfm)$ be a local ring. The map $F:R\rightarrow R$ induces a map on the local cohomology modules $H^j_\mfm(R)$, which is referred to as the \textbf{Frobenius map on local cohomology}, also called a \textbf{Frobenius action}. This map is not $R$-linear but $p$-linear, in that $F(r\xi) = r^pF(\xi)$ for any $r \in R$ and $\xi \in H^j_\mfm(R)$.
\end{defn}

For a detailed discussion of the Frobenius action on $H^j_\mfm(R)$, see \cite[2.1]{Katzman-Sharp}. Like the Frobenius map itself, this Frobenius action on local cohomology can be used to study the singularities of a local ring. One important classical example of a singularity defined in terms of this action is $F$-injective singularities, which were defined by R. Fedder in the 1980's.

\begin{defn} \label{defn: F-inj}
Let $(R,\mfm)$ be a local ring of dimension $d$. Then, $R$ is \textbf{$F$-injective} if, for all $0 \le j \le d$, the Frobenius action $F:H^j_\mfm(R)\rightarrow H^j_\mfm(R)$ is injective. A non-local ring $R$ is $F$-injective if $R_\mathfrak{p}$ is an $F$-injective local ring for all $\mathfrak{p} \in \operatorname{Spec}(R)$
\end{defn}

Opposite to $F$-injective singularities are those for which the Frobenius action is nilpotent. Some iterate of Frobenius may vanish on the lower local cohomology modules but $F^e(H^d_\mfm(R))$ never vanishes for any $e$, see \cite[Lem.~4.2]{Lyubeznik_2006}. However, the largest submodule of $H^d_\mfm(R)$ on which some iterate $F^e$ of the Frobenius action vanishes is always contained in the \textit{tight closure of $0$} in $H^d_\mfm(R)$, defined by $0^*_{H^d_\mfm(R)} = \left\lbrace \left. \xi \in H^d_\mfm(R) \right| \text{ for some } c \neq 0 \text{ and all } e\gg 0, cF^e(\xi)=0\right\rbrace .$ It is well-known that $0^*_{H^d_\mfm(R)}$ is stable under the Frobenius action. Clearly, $0^*_{H^d_\mfm(R)}$ contains the submodule $0^F_{H^d_\mfm(R)} = \left\lbrace \left. \xi \in H^d_\mfm(R) \right| F^e(\xi) = 0 \text{ for some } e \in \mathbb{N} \right\rbrace,$ which is called the \textit{Frobenius closure of $0$} in $H^d_\mfm(R)$. See \cite[Rmk.~2.6]{PQ19} for an in-depth discussion on the comparison of $0^F_{H^d_\mfm(R)}$ and $0^*_{H^d_\mfm(R)}$.

\begin{defn}\label{defn: F-nil}
Let $(R,\mfm)$ be a local ring of dimension $d$. Then, $R$ is \textbf{$F$-nilpotent} if, for all $0 \le j < d$, the Frobenius action $F:H^j_\mfm(R)\rightarrow H^j_\mfm(R)$ is nilpotent, and further, that $F:0^*_{H^d_\mfm(R)}\rightarrow 0^*_{H^d_\mfm(R)}$ is also nilpotent. A non-local ring $R$ is $F$-nilpotent if $R_\mathfrak{p}$ is an $F$-nilpotent local ring for all $\mathfrak{p} \in \operatorname{Spec}(R)$
\end{defn}

The class of $F$-nilpotent rings was first considered much more recently than $F$-injective rings; it was introduced by this name in work of Srinivas-Takagi \cite{ST17}, although it was studied by Blickle-Bondu in \cite{BB05} under the name \textit{close to $F$-rational}.\footnote{$F$-rational singularities are another classical and interesting kind of $F$-singularities -- in fact, an excellent local ring is $F$-rational if and only if it is both $F$-injective and $F$-nilpotent. However, an excellent $F$-rational ring is locally a normal domain, and as we intend to study exclusively non-normal affine semigroup rings in this article, we will avoid discussing $F$-rationality here.} 

Due to the inherent difficulty in computing properties of local cohomology modules, in general checking whether a specific ring has either of these singularity types is also exceptionally difficult. However, for affine semigroup rings, these conditions are both much more readily checkable, as they only depend on computationally amenable properties of the normalization map; see \Cref{cor: wkNiffFinj} and \Cref{cor: weak normalization equals normalization if and only if F-nilpotent} for more details.

Alongside the Frobenius action on local cohomology, $F:R\rightarrow R$ can be used to define singularities in terms of ideal closures. 

\begin{defn}
Let $R$ be a ring of prime characteristic $p>0$ and let $I\subset R$ be an ideal. Then, $I^{\fbp{p^e}}$ is defined to be the ideal generated by $F^e(I)$, that is, if $I=(x_1,\ldots,x_t)$ then $I^{\fbp{p^e}} = (x_1^{p^e},\ldots,x_t^{p^e})$. Then, the \textbf{Frobenius closure of $I$} is the ideal \[ I^F = \left\lbrace x \in R \left| x^{p^e} \in I^{\fbp{p^e}} \text{ for some } e \in\N\right.\right\rbrace .\] The \textbf{Frobenius test exponent of $I$}, written $\fte I$, is the least exponent $e$ such that $(I^F)^{\fbp{p^e}} = I^{\fbp{p^e}}$. 
\end{defn}

For computational purposes, one may hope for an effective upper bound on $\fte I$ independent of $I$. However, Brenner showed in \cite{Brenner} that this is not possible even in nice rings of low dimension. However, around the same time Katzman-Sharp showed in \cite{Katzman-Sharp} that there was a uniform upper bound on the Frobenius test exponents over all parameter ideals for Cohen-Macaulay rings using the Frobenius action on local cohomology. Consequently, for a local ring $R$, we write $\fte R$ for the uniform upper bound on $\fte \mathfrak{q}$ over all ideals generated by systems of parameters $\mathfrak{q}$, if such an upper bound exists. For a ring standard graded over a field, we write $\fte^* R$ for the uniform upper bound on the Frobenius test exponent of homogeneous parameter ideals. 

To our knowledge, no example of a ring with an infinite Frobenius test exponent has been exhibited to date. In general, it is an interesting open question to determine the largest class of rings for which $\fte R$ is finite. For a survey of the history of the Frobenius test exponent problem, see \cite[Sec.~1]{Maddox}. We will show in \Cref{thm: affine semigroup rings have uniformly bounded fte} that all affine semigroup rings have finite Frobenius test exponent, and even a uniform Frobenius test exponent over all ideals! This shows that affine semigroup rings avoid the problem raised by Brenner.

Finally, the Frobenius closure operation is closely related to $F$-injectivity. In particular, if all parameter ideals of $R$ are Frobenius closed (that is, $\fte R = 0$), then $R$ is $F$-injective. The converse does not hold, however; see \cite{QS} for a discussion on the gap between $F$-injectivity and $\fte R = 0$, which is referred to as \textit{parameter $F$-closed} in \textit{loc. cit}. Stronger than $F$-injectivity is \textit{$F$-purity}, which in our context is equivalent to all ideals being Frobenius closed, i.e. $\fte I = 0$ for all $I$. Although not equivalent in general, by \cite[Prop.~6.2]{Bruns2006}, $F$-injective and $F$-pure are equivalent for affine semigroup rings.

\section{The \texorpdfstring{$p$-}{p-}weak normalization of an affine semigroup}

We begin by recalling the seminormalization of an affine semigroup. While the original definition is given in terms of the intersection of seminormal sub-monoids, we present an equivalent formulation that is better suited for our purposes.

\begin{defn}\label{defn: seminormalization} (Exercise 2.11, \cite{BG09})
    Let $A\subset \mathbb{N}^n$ be an affine semigroup and let $\overline{A}$ be its saturation in $\mathbb{N}^n$. The \textbf{seminormalization} of an affine semigroup is the affine semigroup defined by \[
    ^{+}A = \{ \mathbf{v} \in G(A) \mid m\mathbf{v} \in A \text{ and } m' \mathbf{v} \in A \text{ for some coprime } m,m'\in \mathbb{N}\}.
    \] The affine semigroup $A$ is \textbf{seminormal} if $A=\,^{+}A$.
\end{defn}

The seminormalization of an affine semigroup is related to the seminormalization of its corresponding affine semigroup ring. Further, the seminormalization of an affine semigroup has a useful geometric description given by Reid-Roberts in \cite{Reid2001}, which we recall below.

\begin{equation}\label{geometric+M}
    ^+ A = \bigcup_{\substack{F \text{ is a} \\\text{face of } C(A)}} G(A\cap F) \cap \int(F).
\end{equation}

We now turn to our definition of the $p$-weak normalization of an affine semigroup. 

\begin{defn}\label{defn: p-weak normalization}
    Let $A\subset \mathbb{N}^n$ be an affine semigroup. For a prime number $p$, define the \textbf{$p$-weak normalization of $A$} to be the set \[
    \wkp{A} = \left\lbrace \left. \mb v \in \overline{A} \,\right|\, p^e \mb v \in A \text{ for some } e \in \N\right\rbrace.\] For a fixed prime $p$, we will call $A$ \textbf{$p$-weakly normal} if $\wkp{A}=A$.
\end{defn}

Unlike the normalization or the seminormalization of an affine semigroup, our definition of the weak normalization necessarily incorporates a fixed prime number. We note that the $p$-weak normalization of an affine semigroup is an affine semigroup, and fits in between the seminormalization and the normalization of $A$.

\begin{prop}
    Let $A$ be an affine semigroup. Then, for any prime number $p$, $\wkp{A}$ is an affine semigroup with $^{+}A \subset \wkp{A} \subset \overline{A}.$ Further, for any two distinct primes $p$ and $q$, we have $^{+}A = \wkp{A} \cap \wkq{A}.$
\end{prop}

\begin{proof}
    For notational convenience, set $B=\wkp{A}$. Clearly $\mb 0 \in B$, and if $\mb v, \mb w \in B$, then $p^{e'} \mb v \in A$ and $p^{e''}\mb w \in A$ for some $e', e''$ in $\N$. Letting $e=\max\{e',e''\}$, we see $p^e(\mb v + \mb w) \in A$, showing $B$ is closed under addition. 
    
    Now note that $B$ is finitely generated. To see this, let $k$ be any field, and we have the chain of inclusions \[
    k[A] \subset k[B]\subset k\left[\,  \overline{A}\,\right] = \overline{k[A]}
    \] where $k[A]$ and $\overline{k[A]}$ are Noetherian rings, and $k[A]\subset \overline{k[A]}$ is a module-finite extension (since $k[A]$ is excellent). Thus, $k[A]\subset k[B]$ is also module-finite, which implies $B$ is finitely generated as a monoid by \cite[Thm.~7.7]{Gilmer}. 

    To see that $^{+}A\subset \wkp{A}$, let $\mb v \in \,^{+}A$ with $m\mbv\in A$ and $n \mbv \in A$ for some coprime $n,m \in \mathbb{N}$. Then, there exists some $k$ such that for all $j \ge k$, $j = a_j n+ b_j m$ for some $a_j, b_j \in \mathbb{N}$ (since $n$ and $m$ generate a numerical semigroup in $\mathbb{N}$). Then, we may take $j=p^\ell =an+bm$ for some $\ell \gg 0$, which gives that $p^\ell\mb v = (an+bm)\mb v = a(n\mbv)+b(m\mb v) \in A$. This shows $\mb v \in \wkp{A}$. 
    
    Finally, if $p$ and $q$ are distinct primes and $\mb v \in \wkp{A} \cap \wkq{A}$, then $p^e\mb v \in A$ and $q^f \mb v \in A$, but $\gcd(p^e,q^f)=1$ so $\mb v\in \,^{+}A$. Of course, $^{+}A \subset \wkp{A}\cap \wkq{A}$, so this shows $^{+}A=\wkp{A}\cap \wkq{A}$.
\end{proof}

\begin{rmk}\label{rmk: finite pi index}
Notably, since $^{*p}A$ is an affine semigroup, it has a finite set of generators; say $^{*p}A$ is generated by the set $\{\mbv_1,\cdots,\mbv_t\}$. Then, for each $\mbv_i$ there is an $N_i$ such that $p^{N_i}\mbv_i\in A$. Letting $N_0 = \max\{N_i\}$, we have $p^{N_0}(^{*p}A) \subset A$. In \Cref{alg2}, we provide an algorithm which computes both $\wkp{A}$ and $N_0$.
\end{rmk}

The reason to consider the $p$-weak normalization of an affine semigroup is that it captures the combinatorial information of the weak normalization in characteristic $p$, as the following proposition illustrates.

\begin{prop}\label{prop: p-weak normalization gives weak normalization}
    Let $A$ be an affine semigroup and $p>0$ a prime integer. Then, if $k$ is a field of characteristic $p$, we have that $\wk{k[A]} = k[\wkp{A}]$.
\end{prop}

\begin{proof}
Note that the residue field extension modulo the maximal monomial ideals induced by the homomorphisms $k[A] \rightarrow \wk{k[A]} \rightarrow k[\,\overline{A}\,]$ must be trivial, as $k[A]$ and $k[\,\overline{A}\,]$ have the same residue field ($k$) modulo their respective maximal monomial ideals.

Now, since $\wk{k[A]} \subset \overline{k[A]} = k\left[\,\overline{A}\,\right]$, elements of $\wk{k[A]}$ are $k$-linear combinations of monomials $\uline{x}^{\mathbf{v}}$ where $\mathbf{v} \in \overline{A}$, so to show that $\wk{k[A]}\subset k[\wkp{A}]$, it suffices to consider monomial elements in $\wk{k[A]}$. For such a monomial $\uline{x}^{\mathbf{v}} \in \wk{k[A]}$, we have $(\uline{x}^{\mathbf{v}})^{p^e} \in k[A]$ for some $e$, and $(\uline{x}^{\mathbf{v}})^{p^e} = \uline{x}^{(p^e\mathbf{v})}$, so $p^e \mathbf{v} \in A$, thus $\mathbf{v}\in \wkp{A} $. This shows $\wk{k[A]} \subset k[\wkp{A}]$.  

The other direction is even simpler; for any monomial $\uline{x}^{\mathbf{v}} \in k[\wkp{A}]$, we have $p^e \mathbf{v} \in A$ for some $e$ so $\uline{x}^{\mathbf{v}} \in \wk{k[A]}$.
\end{proof}

As noted before, for a fixed $A$, the semigroup $\wkp{A}$ depends on $p$.

\begin{ex}\label{ex: pinched veronese different weak normalizations}
    Let $A$ be the affine semigroup given below. \[A = \left\langle (2,0,0),(0,2,0),(0,0,2),(1,1,0),(1,0,1)\right\rangle\subset \N^3\]
    Then, ${^{*2}}A = \overline{A}$ but $\wkp{A} = A$ for $p \neq 2$; see the proof of \cite[Thm.~4.3]{MP}. 
\end{ex}

Given this example, it is natural to ask how many distinct $p$-weak normalizations of a semigroup $A$ may occur. Due to the results \cite[Sec.~6]{Bruns2006}, there can be only finitely many. We will spell this out carefully in \Cref{blr weakly normal}.  

Interestingly, if $\wkp{A}=\overline{A}$ for even one prime $p$, we note that there are two options; either every $q$-weak normalization is the seminormalization for $q\neq p$ or every $q$-weak normalization is the normalization.

\begin{prop}\label{prop: two weak normalizations}
Let $A$ be an affine semigroup and suppose $\overline{A}=\wkp{A}$ for some prime $p$. Then, $\wkq{A} = {}^+A$ for all other primes $q$ or $\wkq{A}=\overline{A}$ for all other primes $q$.
\end{prop}

\begin{proof}
    If $q$ is another prime such that $\wkq{A} =\overline{A}$, then $\overline{A}=\wkq{A}\cap\wkp{A} = {}^+A$, so for any other prime $r$ we have $\overline{A}=\,^{+}A \subset \,^{*r}A \subset \overline{A}$. On the other hand, if $q$ is a prime such that $\wkq{A}\neq \overline{A}$, we have $^{+}A = \wkq{A}\cap \wkp{A} = \wkq{A} \cap \overline{A} = \wkq{A}$.
\end{proof}

\begin{ex}
An affine semigroup $A$ is sometimes called a \textit{$\mathcal{C}$-semigroup} (see \cite{garcíagarcía2016extensionwilfsconjectureaffine}) if $(C(A) \cap \Z^n) \setminus A$ is a finite set. This property generalizes \textit{numerical semigroups}, which are sub-semigroups of $\mathbb{N}$ generated by a collection of relatively prime natural numbers. 

Fix a $\mathcal{C}$-semigroup $A$. Then, for any prime $p$, we have $\wkp{A}=\overline{A}$, since $A$ is a $\mathcal{C}$-semigroup implies $\overline{A}\setminus A$ is a finite set but $\{p^e \mb v \mid e \in \mathbb{N}\}$ is an infinite subset of $\overline{A}.$ In particular, for any prime $p$ and any numerical semigroup $S$, $\wkp{S} = \mathbb{N}$.
\end{ex}

We now show that the containments $^+A \subset \wkp{A} \subset \overline{A}$ may be strict for some values of $p$. Later, we will show that for all but finitely many primes $p$, $^{+}A=\wkp{A}$.

\begin{ex}\label{ex: 2 different weak normalizations}
    Let $A$ be the affine semigroup generated by the columns of the matrix below; denote by $\mbv_i$ the $i$th column of the matrix (which we will write as row vectors). \[
    M=\left(
        \begin{array}{ccccc}
            0 & 1 & 2 & 0 & 0 \\
            1 & 3 & 0 & 2 & 0 \\
            0 & 0 & 0 & 2 & 3
        \end{array}
    \right)
    \]

The cone of $A$ is clearly the first octant $\mathbb{R}_{\ge 0}^3$. Further, the group of $A$ contains $\mbv_1=(0,1,0)$, $(1,0,0) = \mbv_2-3\mbv_1$ and $(0,0,1) = \mbv_5-\mbv_4+2\mbv_1$, so the group of $A$ is $\mathbb{Z}^3$ and the normalization of $A$ is $\mathbb{R}^3_{\ge 0} \cap \mathbb{Z}^3 = \mathbb{N}^3$.

Now note that $A$ is not seminormal. One can see that the vector $(1,1,1)$ is in the seminormalization but not the semigroup since $2(1,1,1) = \mbv_3+\mbv_4$ and $3(1,1,1)=\mbv_2 + \mbv_3 + \mbv_5$. According to the \textit{Seminormalization} package for \textit{Macaulay2} \cite{Seminormalization}, we have that the seminormalization of $A$ is generated by the columns of the matrix below. \[
    M'=\left(
        \begin{array}{cccccccc}
             0 & 2 & 0 & 0 & 1 & 0 & 1 & 1\\
             1 & 0 & 0 & 1 & 1 & 1 & 1 & 1\\
             0 & 0 & 3 & 1 & 0 & 2 & 1 & 2
        \end{array}
    \right)
    \]

With the generators for $^+A$ given above, we can see that $(1,0,0)$ is not in the seminormalization of $A$, but $2(1,0,0)$ is in $A$, so $\mb{a}=(1,0,0) \in \,^{*2}A \setminus \,^{+}A$. Similarly, $\mb b = (0,0,1) \in \,^{*3}A \setminus \,^{+}A$, but $\mb a \not \in \,^{*3}A$ and $\mb b \not \in \,^{*2}A$. This shows that $^{*2}A$ and $^{*3}A$ are neither the seminormalization nor the normalization. In fact, by \Cref{prop: two weak normalizations}, $\wkp{A} \neq \overline{A}$ for any $p$. Later, in \Cref{ex: 2 different weak normalizations part 2}, we will compute $\wkp{A}$ for $p=2,3,5$.
\end{ex}


Along the lines of the geometric description of the seminormalization given by Reid-Roberts in \ref{geometric+M}, we provide a similar description of the $p$-weak normalization. For $F$ a face of the cone of an affine semigroup $A$, we will write $A_F$ for $A\cap F$.

\begin{thm}\label{thm: geometric description of p-weak normalization}
   Let $A\subset \mathbb{N}^n$ be an affine semigroup and fix a prime $p$. Recall there is an $N_0$ such that $p^{N_0}(\wkp{A}) \subset A$. Then \[
        \wkp{A} = \bigcup_{\substack{F \text{ is a}\\ \text{face of } C(A)}} 
           \left( \quot{\frac{A_F}{p^{N_0}}} \cap \int(F)\right).
   \] 
\end{thm}

\begin{proof}
    For notational convenience, given a face $F$ of $C(A)$, write \[
        S_F = \quot{\frac{A_F}{p^{N_0}}} \cap \int(F).
    \]
First, suppose $\mbv \in \overline{A}$ and $p^{N_0} \mbv \in A_F$. Since $A_F \subset A$, we have $\mbv \in \wkp{A}$. Now suppose $\mbv \in \wkp{A}$. Then, $\mbv \in \overline{A}$, so $\mbv$ is in the interior of a unique face $F$ of $C\left(\overline{A}\right) = C(A)$. Furthermore, since $\wkp{A}\subset G(A)$, we have $\mbv \in G(A) \cap\int(F)$. Finally, since $F$ is convex and $\mbv$ is in the interior of $F$, we have $p^{N_0}\mbv \in A \cap F$, which shows $\mbv \in \quot{A_F/p^{N_0}}$ for some face $F$ of $C(A)$. Thus, $\wkp{A}\subset  \cup_F S_F$.
\end{proof}

\section{Seminormality, Weak Normality, and \texorpdfstring{$F$}{F}-singularities}

We now turn to understanding the $F$-singularities of affine semigroup rings using our $p$-weak normalization. In particular, if $k$ is a field of prime characteristic $p>0$, computation of $\wkp{A}$ is indispensable in understanding the $F$-singularities of $k[A]$. 

By work of K. Schwede \cite{Sch09} and Datta-Murayama \cite{DM}, it is known that $F$-injective rings are weakly normal. Furthermore, the converse is true for affine semigroup rings due to work of Bruns-Li-R\"omer, which we recall in part below.

\begin{thm}[Prop.~6.2, \cite{Bruns2006}]\label{BLRmain}
    Let $A\subset \mathbb{N}^m$ be a positive affine semigroup which is seminormal, and let $k$ be a field of characteristic $p>0$. Then, the following are equivalent.
\begin{enumerate}[label=(\roman*)]
    \item The prime ideal $(p)$ of $\mathbb{Z}$ is not associated to the $\mathbb{Z}$-module $G(A)\cap \mathbb{R}F / G(A\cap F)$ for any face $F$ of $C(A)$
    \item $k[A]$ is $F$-pure.
    \item $k[A]$ is $F$-injective.
\end{enumerate}
\end{thm}

Though it is well-known to the experts that the first condition above is equivalent to $k[A]$ being weakly normal (see, \cite[Prop.~5.3]{Yasuda}), we reprove this using our $p$-weak normalization.

\begin{prop}\label{blr weakly normal}
    Let $A$ be an affine semigroup and $p$ a prime number. Then, $(p) \in \operatorname{Ass}_\mathbb{Z}(G(A)\cap \mathbb{R}F/G(A\cap F))$ if and only if $\wkp{A} \neq \,^{+}A$. In particular, there are only finitely many primes $p$ for which $\wkp{A} \neq \,^{+}A$.
\end{prop}

\begin{proof}
    Suppose $\mb v \in \wkp{A}\setminus \,^{+}A$. Since $\mb v \in \overline{A}$, there is a unique face of $C(A)$ such that $\mbv \in \int F$. Certainly $\mb v \in G(A)\cap F \subset G(A)\cap \mathbb{R}F$, but we must have that $\mbv \not \in G(A\cap F)$ as otherwise $\mbv \in \,^{+}A$. Then, there is a minimal $e$ so that $p^e\mbv \in A \subset G(A\cap F)$, and so $(p) = \operatorname{Ann}_\mathbb{Z}(p^{e-1}\mbv + G(A\cap F))$.

    Now suppose $\wkp{A}=\,^{+}A$, and pick a maximal face $F$ of the cone of $A$ so that $(p)$ is an associated prime to $G(A)\cap \mathbb{R}F/G(A\cap F)$. Let $\mbv \in G(A)\cap \mathbb{R}F\setminus G(A\cap F)$ with $p\mbv \in G(A\cap F)$. Now, if $\mb v \in \int F$ then \[\mbv \in \frac{G(A\cap F)}{p} \cap \int(F)\] and $\mbv \in G(A)$ to begin with, so by \Cref{thm: geometric description of p-weak normalization}, $\mb v \in \wkp{A}=\,^{+}A$. But then $\mb v \in G(A\cap F)$, which is a contradiction. Thus, $(p)$ is not associated to $G(A)\cap \mathbb{R}/G(A\cap F)$ for any face $F$ of the cone of $A$.
    
    If $\mbv \not\in \int F$, we will replace $\mbv$ with a new vector $\mbv'$ as follows. Write $A = \left\langle \mb a_1,\ldots,\mb a_m\right\rangle$. After possibly re-indexing, we may assume $F = \mathbb{R}_{\ge 0}\{\mb a_1,\ldots,\mb a_f\}$.  Then, as $\mb v \in \mathbb{R}F$, we can write $\mb v = \sum_{i=1}^f b_i \mb a_i$ where $b_i \in \mathbb{R}$ for each $i$. After possibly re-indexing again, we can assume that the $b_i$ are positive for $1 \le i < s$ and that $b_j$ is non-positive for $s \le j \le f$. Now for each $s \le j \le f$, set $c_j = \lceil -b_j\rceil+1$ and $\mb w = \sum_{j=s}^f c_j \mb a_j$; notably, $\mb w \in A \cap F\subset G(A\cap F)$. Now we let $\mb v' = \mb v + \mb w$ so that \begin{itemize}
        \item $\mb v' \not \in G(A\cap F)$, as otherwise, $\mb v = \mb v' - \mb w \in G(A\cap F)$, 
        \item $p\mb v' = p\mb v+p\mb w\in G(A\cap F)$,
        \item and $\mb v' \in \int F$, since each coefficient in the expression for $\mb v'$ as a linear combination of the generators of $F$ are strictly positive.
    \end{itemize}

    Then we may repeat the argument above replacing $\mb v$ with $\mb v'$ to see the result in the case that $\mb v \not \in \int F$.  The final claim follows from the fact that $G(A)$ is  finitely-generated, so $G(A)\cap \mathbb{R} F /G(A\cap F)$ can only have finitely many associated primes as a $\mathbb{Z}$-module.
\end{proof}

\begin{cor}\label{cor: wkNiffFinj}
    Let $A$ be an affine semigroup and let $k$ be a field of prime characteristic $p>0$. Then, $\wkp{A}=A$ if and only if $k[A]$ is $F$-injective.
\end{cor}

\begin{proof}
    If $k[A]$ is $F$-injective, then by \cite[Thm.~4.7]{Sch09}, $k[A]$ is weakly normal as $k[A]$ is excellent, and thus, has a dualizing complex. Hence, $^*k[A]=k[\wkp{A}] = k[A]$, so $\wkp{A}=A$. Conversely, if $\wkp{A}=A$, then $A=\,^{+}A$, so $A$ is seminormal. Thus, by the previous proposition and \cite[Prop.~6.2]{Bruns2006}, $k[A]$ is $F$-injective.
    \end{proof}

As a dual result to that of \cite[Thm.~4.7]{Sch09} and the corollary above, in \cite[Thm~3.5, Cor.~4.6]{DMP} the authors show that an $F$-nilpotent ring $R$ has $\wk{R}=\overline{R}$, and the converse holds for affine semigroup rings. In particular, we have the following.

\begin{cor} \label{cor: weak normalization equals normalization if and only if F-nilpotent}
Let $A$ be an affine semigroup and let $k$ be a field of prime characteristic $p>0$. Then, $\wkp{A}=\overline{A}$ if and only if $k[A]$ is $F$-nilpotent. 
\end{cor}

Combining the results of this section with \Cref{prop: two weak normalizations}, we obtain the following.

\begin{cor}\label{cor: two weak normalizations f-singularities consequences}
Let $A$ be an affine semigroup and let $k$ be a field of prime characteristic $p>0$. If $k[A]$ is $F$-nilpotent, then $\ell[\,^{+}A]$ is $F$-injective for all $q\neq p$ prime and fields $\ell$ of characteristic $q$ or $\ell[A]$ is $F$-nilpotent for all $q\neq p$ prime and fields $\ell$ of characteristic $q$.
\end{cor}

As an additional application of our $p$-weak normalization construction, we see that Frobenius test exponent of all ideals in an affine semigroup ring is uniformly bounded. 

\begin{thm}\label{thm: affine semigroup rings have uniformly bounded fte}
Let $A$ be an affine semigroup and let $k$ be a field of prime characteristic $p>0$. Further, let $N_0$ be minimal so that $p^{N_0}(\wkp{A}) \subset A$. Then, for any ideal $I\subset k[A]$, $\fte I\le N_0$. In particular, $\fte^* k[A]\le N_0$.
\end{thm}

\begin{proof}
Since $k[\wkp{A}]$ is $F$-pure by \cite[Prop.~6.2]{Bruns2006}, so all ideals of $k[\wkp{A}]$ are Frobenius closed. Then, for any ideal $I\subset k[A]$ and $x \in I^F$, we have $x \in I^Fk[\wkp{A}] \subset (Ik[\wkp{A}])^F = Ik[\wkp{A}]$. Writing $I=(y_1,\ldots,y_t)$, we have an expression $x = \sum_{i=1}^t s_i y_i$ where $s_i \in k[\wkp{A}]$, so \[x^{p^{N_0}} = \sum_{i=1}^t s_i^{p^{N_0}}y_i^{p^{N_0}}\] where $s_i^{p^{N_0}} \in k[A]$ for all $i$ as $k[A]\rightarrow k[\wkp{A}]$ is purely inseparable. Hence, $x^{p^{N_0}} \in I^{\fbp{p^{N_0}}}$, so $\fte I \le N_0$.
\end{proof}

This result is a strict improvement to \cite[Cor.~4.6]{DMP} when only considering the Frobenius closure, which assumes that $k[A]$ is $F$-nilpotent to obtain a similar bound. In the next section, we present an algorithm to compute $\wkp{A}$ which also computes the bound above on the Frobenius test exponent.

\section{An algorithm to compute the \texorpdfstring{$p$}{p}-weak normalization}

Our goal now is to give an effective and implementable algorithm to find generators for $\wkp{A}$ given the generators for $A$. We will first present a modification of \cite[Alg~1]{Onthequotient}, which returns a generating set of $A/m$ given generators of $A$, to one which returns a generating set of $\quot{A/m}$ given a generating set of $A$ and $\overline{A}$.

Let $A\subset \N^n$ be an affine semigroup minimally generated by the columns of the $n\times s$ matrix  $L_0$, similarly $\overline{A}$ generated by the columns of the $n\times t$ matrix $M_0$. To compute $\quot{A/m}$, we first create the $n\times (s+t)$ matrix $N = (L_0|(-m)\cdot M_0)$. We then find the affine semigroup $C$ of solutions to the $\N$-linear system of equations \[
N \left(
\begin{array}{c}
\mathbf{\lambda} \\
\mathbf{x}
\end{array}\right) = \mathbf{0}
\] where $\mathbf{\lambda}\in \N^s$ and $\mathbf{x}\in \N^t$ are viewed as vectors of variables. A vector $(\mathbf{\lambda},\mathbf{x})^T$ is in $C$ if and only if $L_0\mathbf{\lambda} = m \cdot M_0 \mathbf{x}$. Since the left hand side is a vector in $A$, we know $m\cdot (M_0 \mathbf{x}) \in A$, or equivalently, $M_0 \mathbf{x} \in A/m$. However, clearly $M_0 \mathbf{x} \in \overline{A}$ as well.  

Hence, given a minimal generating set \[\left\lbrace(\mb \lambda_1,\mb x_1)^T,\ldots, (\mb \lambda_s,\mb x_s)^T\right\rbrace\] of $C$, we have that $\{M_0\mb x_1,\ldots, M_0\mb x_s\}$ generates $\quot{A/m}$, and can be pared down to a minimal generating set thereof. 

\begin{alg}\label{alg1} Compute a minimal generating set of $\quot{A/m}$.\\ \rule{6.5in}{1pt}\vspace{.1in} 
    
    Input: A matrix $L_0$ of the $t$ minimal generators of $A\subset \mathbb{N}^n$ as columns, a matrix $M_0$ of the $u$ minimal generators for $\overline{A}\subset\mathbb{N}^n$, and a positive integer $m$.
    
    Output: A matrix $L_1$ whose columns are a minimal generating set for $A/m\cap G(A)$.
    \begin{enumerate}
        \item Let $N$ be the matrix $(L_0|(-m)M_0).$
        \item Produce the minimal generating set of the $\mathbb{N}$-valued solutions to the matrix equation \[N\left( \begin{array}{c}
        \mathbf{\lambda} \\
        \mb x
        \end{array}\right) = \mb 0\] where $\mb \lambda = (\lambda_1,\ldots,\lambda_t)$ and $\mb x = (x_1,\ldots,x_u)$ are vectors of variables.
        \item Project the solutions in the previous step to the last $u$ coordinates, store them as columns of a matrix $L_1$.
        \item Remove any columns of $L_1$ that are in the semigroup generated by the remaining columns.
        \item Return $L_1$.
    \end{enumerate}

    We will call this algorithm via the formula $q(L_0,M_0,m) = L_1$.
    
    \noindent \rule{6.5in}{1pt}
\end{alg}

To apply this algorithm to help us construct $\wkp{A}$, note that $\wkp{A} = \quot{A/p^{N_0}}$ for all $N_0$ large enough. In particular, we need only compute until \[\quot{A/p^{N}}=\quot{A/p^{N+1}},\] and take $N_0 = N$. Further, to find the matrix $M_0$, one can use the \verb|integralClosure| method in \textit{Macaulay2} \cite{M2} or find a Hilbert basis using \textit{Normaliz} \cite{Normaliz}.\\

\begin{alg}\label{alg2} Compute a minimal generating set of $\wkp{A}.$ \\ \rule{6.5in}{1pt} 

Input: A matrix $L_0$ of the minimal generators of $A$ as columns, a matrix $M_0$ of the minimal generators for $\overline{A}$, and a prime integer $p$.

Output: A matrix $L_{N_0}$ whose columns are a minimal generating set for $\wkp{A}$, and $N_0$ minimal so that $p^{N_0}(\wkp{A})\subset A$.
\begin{enumerate}
    \item Use \Cref{alg1} to obtain $L_1 = q(L_0,M_0,p)$. Check if the columns of $L_1$ are inside $\left\langle L_0 \right\rangle$. If so, break and return $(L_0,0).$ If not, continue. \[\vdots\]
     
    \item[(s+1)] Obtain $L_{s+1} = q(L_{s},M_0,p)$. Check if $L_{s+1} \in \left\langle L_{s}\right\rangle$. If so, break and return $(L_{s},s)$. If not, continue. 
\end{enumerate}

Since there is an $N_0$ such that $\quot{A/p^{N_0}}=\quot{A/p^{N_0+1}}$, this terminates at stage $N_0+1$ and the algorithm returns $(L_{N_0},N_0)$. 

\noindent \rule{6.5in}{1pt} 
\end{alg}

We now give an example implementation of the algorithms above. The code is written in SageMath \cite{sagemath} and uses the integer programming software 4ti2 \cite{4ti2} in a crucial way to solve the integer programming problems which underpin semigroup membership testing.

\begin{verbatim}
def membershiptest(A,v):
    #Semigroup membership test
    W = four_ti_2.zsolve(A,["="]*A.nrows(),v,[1]*A.ncols());
    if not W[0]:
        return False
    else:
        return True
    
def subsemigrouptest(A,B):
    #Sub-semigroup test
    for i in range(A.ncols()):
        if not membershiptest(B,A.T[i]):
            return False    
    return True

def reducegenset(A):
    #Reduces a generating set to a minimal generating set
    eliminate = [];
    #Delete 0 columns
    for i in range(A.ncols()):
        if matrix(A.T[i]) == matrix([0]*A.nrows()):
            eliminate.append(i)
    A = A.delete_columns(eliminate);
    eliminate = [];
    #Delete repeat columns
    for i in range(A.ncols()-1):
        for j in range(i+1,A.ncols()):
            if A.T[i] == A.T[j]:
                eliminate.append(j)
    A = A.delete_columns(eliminate);
    eliminate = [];
    #Delete redundant semigroup gens
    for i in range(A.ncols()):
        B = A.delete_columns([i]);
        if subsemigrouptest(A,B):
            eliminate.append(i);
    return A.delete_columns(eliminate)

def quotientbypin(A,B,p):
    #Runs Algorithm 1
    L = list(IntegerRange(A.ncols(),A.ncols()+B.ncolS()));
    C = A.augment((-p)*B);
    W = four_ti_2.zsolve(C,["="]*C.nrows(),[0]*C.nrows(),[1]*C.ncols());
    D = W[1].T.matrix_from_rows(L);
    return reducegenset(B*D)

def wkp(A,B,p): 
    #Runs Algorithm 2
    steps = 0;
    genset0 = A;
    genset1 = quotientbypin(A,B,p); 
    while(True):
        if subsemigrouptest(genset1,genset0):
            return [matrix(genset0), steps]
        else:
            steps = steps+1;
            genset0 = genset1;
            genset1 = quotientbypin(genset0,B,p);
            continue
\end{verbatim} 

We note that this implementation becomes slower as $p$ increases, due to the fact that semigroup membership testing (and solving the underlying integer programming problem) is known to be NP-complete, and the possible solution space grows dramatically as $p$ increases.

\begin{ex}\label{ex: 2 different weak normalizations part 2}
    Let $A$, $M$ be as in \Cref{ex: 2 different weak normalizations}, and recall $G(A)=\mathbb{Z}^3$. Let $I$ be the $3\times 3$ identity matrix, whose columns generate $\overline{A}$. Then, \[
    \verb|wkp(M,I,2)|= \left(\left(\begin{array}{ccccc}
0 & 0 & 1 & 0 & 0\\
0 & 1 & 0 & 1 & 1\\   
3 & 0 & 0 & 2 & 1
\end{array}\right),\, 2\right) \hspace{12.5pt}\text{and}\hspace{12.5pt} \verb|wkp(M,I,3)| = \left(\left(\begin{array}{cccc}
2 & 0 & 0 & 1\\   
0 & 0 & 1 & 1\\  
0 & 1 & 0 & 0
\end{array}\right),\, 1\right).
    \] This tells us $^{*2}A = A/2^2$ and $^{*3}A = A/3$. Furthermore, we see \[
    \verb|wkp(M,I,5)| =  \left(\left(
        \begin{array}{cccccccc}
             2 & 0 & 0 & 0 & 1 & 0 & 1 & 1\\
             0 & 1 & 0 & 1 & 1 & 1 & 1 & 1\\
             0 & 0 & 3 & 1 & 0 & 2 & 1 & 2
        \end{array}
    \right), 2\right)
    \] and the matrix on the left has the generators of the seminormalization of $A$ as its columns. 
\end{ex}

We can use the generators of $\wkp{A}$ obtained from \Cref{alg2} and \textit{Macaulay2} to compute Frobenius closure in $\mathbb{F}_p[A]$ via \Cref{thm: affine semigroup rings have uniformly bounded fte}.

\begin{ex}\label{ex: 2 different weak normalizations part 3}
Let $A$ be the semigroup from \Cref{ex: 2 different weak normalizations}. We will consider the Frobenius closure of some ideals in $Q=\mathbb{F}_2[A]$ and $R=\mathbb{F}_3[A]$. By \Cref{ex: 2 different weak normalizations part 2}, for all ideals $I\subset Q$ we have $\fte I \le 2$ and for all ideals $J\subset R$ we have $\fte J \le 1$. Using the following code in \textit{Macaulay2}, we are able to compute the Frobenius closure of ideals in $Q$ and $R$. The key method is the preimage function, since \verb|preimage(F^e,F^e(K))| is  $K^F$ for $e$ at least $\fte K.$

\begin{verbatim}
 i1 : p = 2
 i2 : k = ZZ/p
 i3 : T = k[a,b,c]
 i4 : S = k[r,s,t,u,v]
 i5 : phi = map(T,S,{b,a*b^3,a^2,b^2*c^2,c^3})
 i6 : Q = S/ker(phi)
 i7 : F = map(Q,Q,{r^p,s^p,t^p,u^p,v^p})
 i8 : K = ideal(r^2+s^2,u^2,t^2+v^2)
 i9 : preimage(F,F(K)) == preimage(F^2,F^2(K))
 o9 : False
i10 : preimage(F^2,F^2(K)) == preimage(F^3,F^3(K)
o10 : True
\end{verbatim}

The output above indicates \verb|preimage(F^2,F^2(K)) = (s,u^2,ru,t^2+v^2,r^2)| is the Frobenius closure of $K$. We can repeat the same code with $p=3$.

\begin{verbatim}
i1 : p = 3
i2 : k = ZZ/p
i3 : T = k[a,b,c]
i4 : S = k[r,s,t,u,v]
i5 : phi = map(T,S,{b,a*b^3,a^2,b^2*c^2,c^3})
i6 : R = S/ker(phi)
i7 : F = map(R,R,{r^p,s^p,t^p,u^p,v^p})
i8 : J = ideal(r^3+s^3,u^3,t^3+v^3)
i9 : preimage(F,F(J)) == preimage(F^2,F^2(J))
o9 : True
\end{verbatim}

This indicates \verb|preimage(F,F(J)) = (u^2,su,ru,s^2,rs,t^3+v^3,r^3)| is the Frobenius closure of $J$. A similar calculation with $p=5$ and using the generators of the seminormalization instead will demonstrate that $\mathbb{F}_5[\,^{+}A]$ has all ideals Frobenius closed, as predicted by \cite[Prop.~6.2]{Bruns2006} and \Cref{cor: wkNiffFinj}. 
\end{ex}

\bibliographystyle{alpha}
\bibliography{maddoxsinghcitations.bib}
\end{document}